\documentclass[a4paper,11pt,reqno]{amsart}
\usepackage[latin1]{inputenc}
\usepackage[english]{babel}
\usepackage{amsmath, amsthm, amssymb, amsopn, amsfonts, amstext, stmaryrd, enumerate, color, mathtools, hyperref, mathrsfs, enumitem, url, comment}
\usepackage[final]{showkeys}
\numberwithin{equation}{section}

\usepackage[
  hmarginratio={1:1},     
  vmarginratio={1:1},     
  textwidth=16cm,        
  textheight=21cm,
  heightrounded,          
]{geometry}

\newtheorem{theorem}{Theorem}
\newtheorem{corollary}[theorem]{Corollary}
\newtheorem{proposition}[theorem]{Proposition}
\newtheorem{lemma}[theorem]{Lemma}
\theoremstyle{definition}
\newtheorem{remark}[theorem]{Remark}

\newtheorem{definition}[theorem]{Definition}

\numberwithin{theorem}{section}
\numberwithin{equation}{section}

\newcommand{\beq}{\begin{equation}}
\newcommand{\eeq}{\end{equation}}

\newcommand{\E}{\mathcal{E}}
\newcommand{\HH}{\mathbb{H}}

\newcommand{\Q}{\mathcal{Q}}
\newcommand{\R}{\mathbb{R}}

\newcommand{\eps}{\varepsilon}

\newcommand{\dist}{{\mbox{\normalfont dist}}}

\newcommand{\ssubset}{\subset\joinrel\subset}

\DeclareMathOperator*{\essinf}{ess\,inf}

\mathtoolsset{showonlyrefs}

\usepackage{todonotes}

\title[Maximum principles and Hopf's lemmas for the Logarithmic Laplacian]{Antisymmetric maximum principles and Hopf's lemmas for the Logarithmic Laplacian, with applications to symmetry results}

\author[L. Pollastro and N. Soave]{Luigi Pollastro and Nicola Soave}
\address[Luigi Pollastro and Nicola Soave]{Dipartimento di Matematica, Universit\`a degli Studi di Torino, Via Carlo Alberto 10, 10123 Torino (Italy)}
\email{luigi.pollastro@unito.it and nicola.soave@unito.it}


\thanks{Acknowledgements: the authors are supported by the PRIN Project no. 2022R537CS ``Nodal Optimization, NOnlinear elliptic equations, NOnlocal geometric problems, with a focus on regularity ($NO^3$)" - funded by European Union - Next Generation EU within the PRIN 2022 program (D.D. 104 - 02$/$02$/$2022 Ministero dell\'Universit\`a e della Ricerca, Italy). Both the authors thank the INdAM-GNAMPA group (Italy).}

\begin{document}

\maketitle

\begin{abstract}
We prove antisymmetric maximum principles and Hopf-type lemmas for linear problems described by the Logarithmic Laplacian. As application, we prove the symmetry of solutions for semilinear problems in symmetric sets, and a rigidity result for the parallel surface problem for the Logarithmic Laplacian.
\end{abstract}

\section{Introduction}

The purpose of this paper is to demonstrate some symmetry results, as well as the maximum principles and the Hopf-type lemmas necessary for their proofs, for semilinear problems involving the Logarithmic Laplacian $L_\Delta$. The Logarithmic Laplacian is the pseudo differential operator with symbol $2\log|\xi|$ (here and in what follows, $\log$ denotes the natural Logarithm). It can be seen as the first order term in the Taylor expansion of the fractional Laplacian $(-\Delta)^s$ as $s$ goes to $0$, in the sense that, for $\varphi \in C^\infty_c(\R^N)$, 
\[
(-\Delta)^s \varphi = \varphi + s L_\Delta \varphi + o(s) \qquad \text{for $s \to 0^+$, in $L^p(\R^N)$ with $1<p \le \infty$},
\]
see \cite[Theorem 1.1]{chen2019dirichlet}. Moreover, for sufficiently regular functions, one has the pointwise integral representation
\begin{equation}\label{pointwise L D}
\begin{split}
L_\Delta u(x) &= c_N \int_{\R^N} \frac{u(x) \chi_{B_1(x)}(y)- u(y)}{|x - y|^N} \, dy + \rho_N u(x) \\
& = c_N \int_{B_1(x)} \frac{u(x)-u(y)}{|x-y|^N}\,dy - c_N \int_{\R^N \setminus B_1(x)} \frac{u(y)}{|x-y|^N}\,dy + \rho_N u(x),
\end{split}
\end{equation}
where both $c_N, \rho_N$ are positive dimensional constants computed explicitly.

Starting from \cite{chen2019dirichlet}, the operator $L_\Delta$ attracted a lot of attention, in light of both its theoretical interest and its relevance in some applications. From the theoretical point of view, the kernel in \eqref{pointwise L D} is \emph{of zero-order}, in the sense that it is a limiting case for hypersingular integrals. Thus, the regularizing properties are very weak and not completely understood at the moment. We refer the interested reader to \cite{ChLaSa, chen2019dirichlet,FeJa23, hernandez2024optimal} and references therein. From the point of view of the applications, the Logarithmic Laplacian is used to describe the differentiability properties of the solution mapping of fractional Dirichlet problems, and the behavior of solutions to linear and nonlinear problems involving the fractional Laplacian, in the small order limit $s \to 0^+$, see \cite{chen2019dirichlet, JaSaWe, HerSal22}. 

A notable difference with respect to both the standard and the fractional Laplacian is that the operator $L_\Delta$ \emph{does not} satisfy the maximum principle in general sets, see \cite[Corollary 1.10]{chen2019dirichlet}. This is equivalent to the fact that the first eigenvalue of the Dirichlet problem for $L_\Delta$ is not always strictly positive. Some maximum principles and Hopf-type lemmas for positive supersolutions were already established in \cite{chen2019dirichlet} and \cite{hernandez2024optimal}, respectively.
Our aim is to provide analogue results in the antisymmetric setting and apply them to prove two symmetry results for semilinear problems. More precisely, for antisymmetric supersolutions, we prove a general weak maximum principle, a weak maximum principle in sets with small measure, two Hopf-type lemmas, and a strong maximum principle. We also extend these results removing the antisymmetry assumptions (in such case the results are partially already known, and the proofs are in fact easier). Regarding the applications, the first one is the extension in the logarithmic setting to the celebrated Gidas-Ni-Nirenberg symmetry result \cite{GiNiNi}:

\begin{theorem}
\label{GNN_convex_domain}
Let $\Omega \subset \R^N$ be a bounded open set with Lipschitz boundary, convex with respect to the direction $e_1$, and symmetric with respect to $\{ x_1  = 0 \}$. Let also $f: \R \to \R$ be a locally Lipschitz function, and assume that $u \in C(\overline{\Omega}) \cap \HH(\Omega)$ is a strictly positive weak solution of
\begin{equation}
\begin{cases}
L_\Delta u = f(u) &\textrm{in} \ \Omega,\\
u= 0 &\textrm{in} \ \R^N \setminus \Omega.
\end{cases}
\end{equation}
Then, $u$ is symmetric with respect to $\{ x_1 = 0 \}$, and decreasing in the direction $e_1$ in $\Omega \cap \{x_1>0\}$.
\end{theorem}

We refer to Section \ref{section_2} for the definitions of weak solution and the functional space $\HH(\Omega)$.

An immediate application of Theorem \ref{GNN_convex_domain} to the case where $\Omega$ is a ball $B_R$ yields the radial symmetry in $B_R$ for the Logarithmic Laplacian. The analogue result for the fractional Laplacian was proved in \cite{FeWa}.

\begin{corollary}
Let $f: \R \to \R$ be a locally Lipschitz function, and assume that $u \in C(\overline{B_R}) \cap \HH(B_R)$ is a strictly positive weak solution of
\begin{equation}
\begin{cases}
L_\Delta u = f(u) &\textrm{in} \ B_R,\\
u= 0 &\textrm{in} \ \R^N \setminus B_R.
\end{cases}
\end{equation}
Then, $u$ is radial and radially decreasing.
\end{corollary}

The second application concerns the so called \emph{parallel surface problem}, a variant of the Serrin's overdetermined problem. Let $G \subset \R^N$ be a bounded open set, and 
\[
\Omega := G + B_R = \{ x+ y \in \R^N: \ x \in G, \ |y|<R\},
\] 
the Minkowski sum of the set $G$ with a ball of radius $R>0$. Our goal is to investigate the rigidity property of the problem
\begin{equation}\label{semil pb}
\begin{cases}
L_\Delta u = f(u), \quad u>0  &\textrm{in} \ \Omega,\\
u = 0  &\textrm{in} \ \R^N \setminus \Omega,
\end{cases}
\end{equation}
under the overdetermined condition
\begin{equation}
\label{parallel_surface_condition}
u = c \quad \textrm{on} \ \partial G.
\end{equation}

\begin{theorem}
\label{parallel_surface_log_torsion}
Let $G \subset \R^N$ be a bounded open set with $C^1$ boundary and $\Omega := G + B_R$ the Minkowski sum of the set $G$ with a ball of radius $R>0$. Let also $f: \R \to \R$ be a locally Lipschitz continuous function, and assume that $u \in C(\overline{\Omega}) \cap \HH(\Omega)$ satisfies \eqref{semil pb} under the overdetermined condition \eqref{parallel_surface_condition}. Then $G$ and $\Omega$ are concentric balls and $u$ is radial and radially decreasing about the center of both $G$ and $\Omega$.
\end{theorem}

Theorem \ref{parallel_surface_log_torsion} is the version of \cite[Theorem 2.1]{CiMaSa} for the Logarithmic Laplacian\footnote{Actually, Theorem 2.1 is focused on the torsion problem, namely $f(t) =1$.}. We also refer to \cite{CiMaSa15, Shah} for related results, and to \cite{ciraolo2023symmetry} for the fractional Laplacian case.

The strategy for proving the aforementioned results closely follows the classical one. However, to make this adaptation, it is necessary to prove maximum principles and Hopf's lemmas for antisymmetric solutions of semilinear problems described by the Logarithmic Laplacian. This is the content of Section \ref{section_2}, and is the most novel part of the work.

Both in Theorem \ref{GNN_convex_domain} and in Theorem \ref{parallel_surface_log_torsion}, we do not require that $\Omega$ is connected. This is a typical feature of nonlocal problems, already observed, e.g., in \cite{ciraolo2023symmetry, fall2015overdetermined}. Furthermore, and more importantly, we do not make any assumptions about the measure of $\Omega$. This is because the proofs rely solely on the use of maximum principles and Hopf's lemmas on appropriate subsets of $\Omega$, as is customary in the method of moving planes. 

Finally, we point out that the existence of solutions for problem \eqref{semil pb} is not always granted, not even in the special case $f \equiv 1$. If $L_\Delta$ satisfies the maximum principle on $\Omega$, which is equivalent to require that the first eigenvalue of the Dirichlet-log-Laplacian on $\Omega$ is strictly positive (see \cite[Theorem 4.8]{chen2019dirichlet}), then the Poisson problem
\[
\begin{cases}
L_\Delta u = g  &\textrm{in} \ \Omega\\
u = 0  &\textrm{in} \ \R^N \setminus \Omega
\end{cases}
\]
admits a weak solution for any $g \in L^2(\Omega)$, by the Lax-Milgram theorem. The strict positivity of $\lambda_1^L(\Omega)$ is ensured for domains with small measure, as proved in \cite{chen2019dirichlet}. But it can happen that $\lambda_1^L(\Omega) \le 0$ and that the Poisson problem in $\Omega$ does not have a classical solution for some $g$, see \cite[Corollary 1.10]{chen2019dirichlet} and \cite[Theorem 1.1]{ChLaSa}. It should be noted that it is not possible to rescale the problem in such a way as to reduce it to the case of sets with small measure, since the operator does not behave well with respect to rescaling \cite[Appendix A.1]{hernandez2024optimal}.

The structure of the paper is the following: in Section \ref{sec: pre}, we introduce the basic notation, functional spaces, and preliminary results which will be used in the rest of the paper. In Section \ref{section_2}, we state and prove the antisymmetric maximum principles and Hopf lemmas for the Logarithmic Laplacian. Section \ref{sec: GNN} and \ref{sec: parallel} contain the proof of Theorems \ref{GNN_convex_domain} and \ref{parallel_surface_log_torsion}, respectively. Finally, in Appendix A we discuss what changes in the results of Sections 3 and 4 in \cite{chen2019dirichlet} when one considers an open bounded set $\Omega$ instead of a domain, which is needed in order to allow disconnected sets in Theorems \ref{GNN_convex_domain} and \ref{parallel_surface_log_torsion}.

\section{Setting, notation and preliminaries}\label{sec: pre}


Following \cite{chen2019dirichlet, FeKaVo}, we recall the weak formulation of equations and ``boundary value problems" (or, better, exterior data problems) associated with $L_\Delta$. We start by defining $k: \R^N \setminus \{0\} \to \R$ and $j: \R^N \to \R$ by
\begin{equation}\label{def k e j}
k(z) := \frac{c_N}{|z|^N} \chi_{B_1}(z) \quad \textrm{and} \quad j(z):= \frac{c_N}{|z|^N} \chi_{B_1^c} (z),
\end{equation}
and we set
\[
\HH(\R^N):= \left\{u:\R^N \to \R \ \text{measurable}: \ u \in L^2(\R^N) \quad \text{and} \quad \mathcal{E}(u,u)^{1/2}<+\infty\right\},
\]
where
\[
\mathcal{E}(u,v) :=\frac12 \iint_{\R^N \times \R^N} (u(x) - u(y)) (v(x) - v(y)) k(x-y)\, dx dy\,.
\]
Moreover, for any open set $\Omega \subset \R^N$, we consider
\[
\HH(\Omega):= \left\{ \ u \in \HH(\R^N): \ u=0 \ \text{a.e. in $\Omega^c$} \ \right\}.
\]
It is possible to check that $\HH(\R^N)$ is a Hilbert space with inner product 
\[
\langle u, v \rangle_{\HH(\R^N)} := \int_{\R^N} uv\,dx + \mathcal{E}(u,v),
\]
and $[u]_{\HH(\R^N)}:= \mathcal{E}(u,u)^{1/2}$ is a seminorm. Moreover, 
\[
\left(\HH(\Omega), \,\|\cdot\|_{\HH(\R^N)}\right) \hookrightarrow \left(\HH(\R^N) , \,\|\cdot\|_{\HH(\R^N)}\right),
\]
and, by \cite[Lemma 2.7]{FeKaVo}, if $\Omega$ is also bounded we have that 
\[
\inf_{u \in \HH(\Omega) \setminus \{0\}} \frac{\mathcal{E}(u,u)}{\|u\|_{L^2(\Omega)}^2} >0.
\] 
Thus, $\mathcal{E}(u,u)^{1/2}$ is a norm in $\HH(\Omega)$, equivalent to the standard one. Finally, the embedding $\HH(\Omega) \hookrightarrow L^2(\Omega)$ is compact \cite[Theorem 2.1]{CoDP}.

As proved in \cite{chen2019dirichlet}, the bilinear form associated to $L_\Delta$ is
\begin{equation}\label{def EL}
\E_L(u,v) := \E (u,v) - \iint_{\R^N \times \R^N} u(x) v(y) j(x - y) \, dx dy + \rho_N \int_{\R^N} u v\,dx,
\end{equation}
which is well defined for $u, v \in \HH(\Omega)$ whenever $\Omega$ is a bounded open set with Lipschitz boundary\footnote{In \cite{chen2019dirichlet}, it is also assumed that $\Omega$ is connected. On the other hand, to check that $\E_L$ is well defined on $\HH(\Omega) \times \HH(\Omega)$, it is used that $\Omega$ is bounded and that the space of uniformly Dini continuous functions in $\Omega$ with compact support is dense in $\Omega$, which follows from \cite[Theorem 3.1]{chen2019dirichlet}. It is plain that \cite[Theorem 3.1]{chen2019dirichlet} holds also for disconnected set, so that the connectedness assumption on $\Omega$ can be removed.}. In fact, in order to consider functions with non-zero values outside of $\Omega$, we observe that $\E_L(u,v)$ is well defined also for $u \in \HH(\R^N)$ and $v \in \HH(\Omega)$. Indeed, the first and the last term in the definition of $\E_L$ are clearly well defined. As far as the second term is concerned, one has that
\[
u \in L^2(\R^N) \mapsto j * u \in L^2(\Omega)
\]
since $\Omega$ is bounded and $j \in L^r(\R^N)$ for every $r>1$. This allows us to give the following definitions:

\begin{definition}
Let $g \in L^2(\Omega)$. A function $u\in \HH(\R^N)$ is a \emph{(weak) supersolution} (resp. \emph{(weak) subsolution}) of $L_\Delta \varphi = g$ in $\Omega$, namely $L_\Delta u \ge g$ in $\Omega$ (resp. $L_\Delta u \le g$ in $\Omega$), if
\[
\E_L(u,\phi) \ge \int_{\Omega} g \phi\,dx  \quad \bigg(\text{resp.} \quad \E_L(u,\phi) \le \int_{\Omega} g \phi\,dx\bigg) \quad \text{for every nonnegative $\phi \in \HH(\Omega)$.}
\]
It is a \emph{(weak) solution} of $L_\Delta \varphi = g$ in $\Omega$ if it is both a supersolution and a subsolution.
\end{definition}

We also have the useful alternative representation of $\E_L(u,u)$ for functions $u\in \HH(\Omega)$.
\begin{proposition}[\cite{chen2019dirichlet}, Proposition 3.2]
For $u \in \HH(\Omega)$ we have
\begin{equation}
\label{EL_alternative}
\E_L(u,u) = \frac{c_N}{2} \iint_{\Omega \times \Omega} \frac{(u(x) - u(y))^2}{|x-y|^N} \, dxdy + \int_{\Omega} (h_\Omega (x) + \rho_N) u^2(x) \, dx,
\end{equation}
where 
$$
h_\Omega(x):= c_N \left( \int_{B_1(x) \setminus \Omega} \frac{1}{|x-y|^N} \, dy - \int_{\Omega \setminus B_1(x)} \frac{1}{|x-y|^N} \, dy \right).
$$
\end{proposition}


Eigenvalues and eigenfunctions are then defined in a natural way, see \cite[Section 3]{chen2019dirichlet}. The results in \cite[Section 3]{chen2019dirichlet} are stated for bounded \emph{domains} of $\R^N$, but most of them hold true also in general bounded open sets. We discuss this point in the appendix of the present paper. In particular, for any bounded open set $\Omega \subset \R^N$, as in \cite[Theorem 3.4]{chen2019dirichlet} we have that the operator $L_\Delta$ in $\HH(\Omega)$ has an increasing sequence of eigenvalues $\lambda_k^L(\Omega)$ tending to $+\infty$, whose associated eigenfunctions form an orthonormal basis of $L^2(\Omega)$. The first eigenvalue $\lambda_1^L(\Omega)$ is characterized as
\[
\lambda_1^L(\Omega):= \inf\left\{ \E_L(u,u): \ u \in \HH(\Omega) \quad \text{and} \quad \|u\|_{L^2(\Omega)}=1\right\}.
\]
It is useful to have bounds on $\lambda_1^L(\Omega)$, since its strict positivity is equivalent to the validity of a maximum principle \cite[Theorem 4.8]{chen2019dirichlet}. It is remarkable that, differently to what happens for the Laplacian and for the fractional Laplacian, such strict positivity is false on general open bounded sets, see \cite[Corollary 4.9]{chen2019dirichlet}. On the other hand, \cite[Corollary 4.12]{chen2019dirichlet} implies that $\lambda_1^L(\Omega)>0$ on open bounded sets $\Omega$ with small measure. In fact, by combining some results proved in \cite{chen2019dirichlet}, it is also possible to obtain a quantitative lower bound on $\lambda_1^L(\Omega)$ in some cases. This is the content of the following remark.

\begin{remark}
\label{lowerboundlambdaone}
Let $\Omega$ be open and bounded, and let $B_r$ be the ball with the same volume of $\Omega$. By \cite[Lemma 4.11]{chen2019dirichlet}, we have that
\begin{equation}
h_{B_r}(x) \ge 2 \, \log \frac{1}{r} \qquad \forall x \in B_r,
\end{equation}
and therefore, denoting by $u_1$ the $L^2$-normalized eigenfunction corresponding to $\lambda_1^L(B_r)$, and using \eqref{EL_alternative}, a simple computation yields
\begin{equation}
\begin{split}
\lambda_1^L(B_r) &= \lambda_1^L(B_r) \, \| u_1 \|_{L^2(B_r)}^2 = \E_L(u_1,u_1) \\
&= \frac{c_N}{2} \iint_{B_r \times B_r} \frac{(u_1(x) - u_1(y))^2}{|x-y|^N} \, dxdy + \int_{B_r} (h_{B_r} (x) + \rho_N) u_1^2(x) \, dx\\
&\ge \left( 2 \, \log \frac{1}{r} - |\rho_N| \right) \, \|u_1\|^2_{L^2(B_r)} = \left( 2 \, \log \frac{1}{r} - |\rho_N| \right).
\end{split}
\end{equation}
As a consequence, by \cite[Corollary 3.6]{chen2019dirichlet} (which holds true for open sets, not necessarily connected, see the appendix), we deduce that
\[
\lambda_1^L(\Omega) \ge \lambda_1^L(B_r) \ge 2 \, \log \frac{1}{r} - |\rho_N|.
\]
This quantity diverges to $+\infty$ as $r \to 0^+$, namely as the measure of $\Omega$ tends to $0^+$.
\end{remark}

\section{Antisymmetric Weak Maximum Principles and Hopf-type Lemmas}
\label{section_2}

Due to the nonlocal nature of the Logarithmic Laplacian, in order to apply a direct method of the moving planes it is necessary to prove maximum principles and Hopf-type lemmas for antisymmetric functions. This is the content of this section. A major inspiration for the results of this section is \cite{fall2015overdetermined}, we also refer to \cite{BaScMo, chen2017direct, fall2015overdetermined}. 

Let $H \subset \R^N$ be a half-space, $T=\partial H$ its boundary, and $\Q: \R^N \to \R^N$ the map $x \mapsto \bar{x}$, where $\bar{x}$ is the reflection with respect to $T$. A measurable function $u : \R^N \to \R$ is \emph{antisymmetric} with respect to $T$ if $u(\bar x) = - u(x)$ for a.e. $x \in \R^N$.

\begin{proposition}[Antisymmetric weak maximum principle]
\label{lem: wmp}
Let $\Omega \subset H$ be a bounded open set with Lipschitz boundary, and let $u \in \HH(\R^N)$ be an antisymmetric function with respect to $T$, such that
\begin{equation}
\begin{cases}
    L_\Delta u \ge V(x) u \quad &\textrm{in} \ \Omega,\\
u \ge 0 \quad &\textrm{in} \ H \setminus \Omega
\end{cases}
\end{equation}
in a weak sense, where $V \in L^\infty(\Omega)$ is such that $\|V\|_{L^\infty(\Omega)} < \lambda^L_1(\Omega)$. Then $\varphi:= u_- \chi_H \in \HH(\Omega)$, and $u \ge 0$ a.e. in $\Omega$.
%
%
\end{proposition}

\begin{remark}
It is implicitly assumed that $\lambda^L_1(\Omega)>0$.
\end{remark}

\begin{proof}
It is easy to see that $\varphi\in \HH (\Omega)$. We are then able to compute $\E_L(u, \varphi)$ and, since $u$ is a supersolution, we obtain
\begin{equation}
\label{PS1s3eq21}
   \int_{\Omega} V u \varphi \, dx  \le  \E(u,\varphi) -  \iint_{\R^N \times \R^N} u(x) \varphi(y) j(x-y) \, dx dy  + \rho_N \int_{\R^N} u \varphi \, dx.
\end{equation}
Now we look for estimates of each summand in the right-hand side of \eqref{PS1s3eq21}. Concerning the latter term, we immediately see that 
\begin{equation}\label{t3neg}
\rho_N \int_{\R^N} u \varphi \, dx = - \rho_N \int_{\R^N} \varphi^2 \, dx \le 0.
\end{equation}

To handle the first term in \eqref{PS1s3eq21}, we make use of an identity first appeared in the proof of \cite[Proposition 3.1]{fall2015overdetermined}. For a.e. $x,y \in \R^N$ we have
\begin{equation}
\begin{split}
(u(x) - u(y))(\varphi(x) - \varphi(y)) &+ (\varphi(x) - \varphi(y))^2 = - \varphi(x) (\varphi(y) + u(y)) - \varphi(y) (\varphi(x) + u(x)).
\end{split}
\end{equation}
Therefore, using the antisymmetry of $u$, we deduce that
\begin{equation}
\label{PS1s3eq26}
\begin{split}
\E(u,\varphi) &= - \E(\varphi, \varphi) -  \int_{\R^N} \varphi(y) \int_{\R^N} (\varphi (x) + u(x)) k(x - y) \, dx \, dy \\
&= - \E(\varphi, \varphi) -  \int_{\R^N} \varphi(y) \int_{\R^N} (\varphi (x) + u(x)) \left( \chi_H(x) + \chi_{H^c}(x) \right) k(x - y) \, dx \, dy\\
&= - \E(\varphi, \varphi) -  \int_{\R^N} \varphi(y) \left( \int_{H} u_+(x) k(x - y) \, dx - \int_{H} u(x) k(\bar{x} - y) \, dx \right) dy\\
&\le - \E(\varphi, \varphi) -  \int_{\R^N} \varphi(y) \int_{H} u_+(x) \left(  k(x - y) - k(\bar{x} - y) \right) \, dx \, dy.
\end{split}
\end{equation}

Finally, to estimate the second term in \eqref{PS1s3eq21}, we observe that
\[
u \ge u_+ \chi_H - u_- \chi_H - u_- \chi_{H^c},
\]
and that by antisymmetry $u_-(\bar{x}) = u_+(x)$ for every $x \in H$. Therefore
\begin{equation}
\label{PS1s3eq27}
\begin{split}
- \iint_{\R^N \times \R^N} u(x)\varphi(y) j(x-y) \, dx dy  \le  \iint_{\R^N \times \R^N} \varphi(x) \varphi(y) j(x-y) \, dx dy& \\
 - \int_{\R^N} \varphi(y) \int_{H} u_+(x) \left(  j(x - y) - j(\bar{x} - y) \right) \, dx \, dy&.
\end{split} 
\end{equation} 

We are now ready to come back to \eqref{PS1s3eq21}: by \eqref{t3neg}, \eqref{PS1s3eq26} and \eqref{PS1s3eq27}, and observing that
\begin{multline}
\label{PS1s3eq31}
 \int_{\R^N} \varphi(y) \int_{H} u_+(x) \left(  k(x - y) + j(x - y) - k(\bar{x} - y) - j(\bar{x} - y) \right) \, dx \, dy \\
= c_N \int_{H} \varphi(y) \int_{H} u_+(x) \left( \frac{1}{|x-y|^N} - \frac{1}{|\bar{x}-y|^N} \right) dx \, dy \ge 0
\end{multline}
(since $|x-y| \le |\bar x-y|$ for every $x, y \in H$), from \eqref{PS1s3eq21} we obtain
\begin{equation}
\begin{split}
-\int_{\Omega} V \varphi^2 \,dx & \le -\E(\varphi,\varphi) + c_N \iint_{\R^N \times \R^N} \varphi(x) \varphi(y) j(x-y) \, dx dy - \rho_N \int_{\R^N} \varphi^2 \, dx \\
&= - \E_L(\varphi, \varphi).
\end{split}
\end{equation}
Recalling the definition of the first eigenvalue $\lambda_1^L(\Omega)$, this implies that 
\[
(\lambda_1^L(\Omega)  - \|V\|_{L^\infty(\Omega)}) \|\varphi\|_{L^2(\Omega)}^2 \le 0,
\]
which, in view of our assumption on $V$, gives the desired result.
\end{proof}

It is convenient to state a direct consequence of Proposition \ref{lem: wmp} and Remark \ref{lowerboundlambdaone}.

\begin{corollary}[Antisymmetric maximum principle in sets with small measure]
\label{cor: awmp small domains}
Let $V_\infty \ge 0$. There exists $\delta>0$ depending only on $V_\infty$ and on the dimension $N$ such that, if $\Omega \subset H$ is a bounded open set of $\R^N$ with Lipschitz boundary, if $|\Omega|<\delta$, and if $u \in \HH(\R^N)$ is an antisymmetric function with respect to $T$ such that
\begin{equation}
\begin{cases}
    L_\Delta u \ge V(x) u \quad &\textrm{in} \ \Omega,\\
u \ge 0 \quad &\textrm{in} \ H \setminus \Omega,
\end{cases}
\end{equation}
where $V \in L^\infty(\Omega)$ and $\|V\|_{L^\infty(\Omega)} \le V_\infty$, then $u \ge 0$ a.e. in $\Omega$.
\end{corollary}

\begin{proof}
By Proposition \ref{lem: wmp}, it is sufficient to check that $\lambda_1^L(\Omega) > V_\infty$ provided that $|\Omega|$ is small enough. This is ensured by Remark \ref{lowerboundlambdaone}.
\end{proof}

\begin{lemma}[Antisymmetric Hopf's lemma for $L_\Delta$]
\label{antisymmetric_Hopf_logarithmic_Laplacian}
Let $\Omega \subset H$ be a bounded open set with Lipschitz boundary, let $B \subset \Omega$ be a ball such that $B \ssubset H$ and $\lambda_1^L(B)>0$, and let $u \in \HH(\R^N)$ be an antisymmetric function with respect to $T$ such that
\begin{equation}
\begin{cases}
    L_\Delta u \ge V(x) u  \quad &\textrm{in} \ \Omega,\\
u \ge 0 \quad &\textrm{in} \ H,
\end{cases}
\end{equation}
where $V \in L^\infty(\Omega)$ is such that $\|V\|_{L^\infty(\Omega)} \le \lambda^L_1(B)$.
Let also $B' \ssubset H \setminus \overline{B}$ be another ball, such that $\essinf_{B'} u > 0$. Then, there exists a constant $C > 0$ depending on $N$, $B$, $B'$, $\essinf_{B'} u$ and $\| V \|_{L^\infty(\Omega)}$, such that
\begin{equation}
\label{PS1s3eq42}
u(x) \ge C \, \ell^{1/2} (\dist (x,\partial B)) \quad \textrm{for a.e.} \ x \in B,
\end{equation}
where
\beq\label{def ell}
\ell(r):= \frac{1}{|\log(\min\{r,0.1\})|}.
\eeq
Moreover, if $u\in C(\overline{B})$ and $u(x_0) = 0$ for some $ x_0 \in \partial B \cap \partial \Omega$, then we have
\begin{equation}
\liminf_{t \searrow 0} \frac{u(x_0 - t \nu(x_0)) }{\ell^{1/2}(t)} > 0,
\end{equation}
where $\nu$ is the outer unit normal vector field along $\partial B$.
\end{lemma}

\begin{remark}
As already pointed out, the assumption $\lambda_1^L(B)>0$ is satisfied whenever $B$ is sufficiently small.\end{remark}

\begin{proof}
Since $\lambda_1^L(B)>0$, it is well defined the log-torsion function $\psi_B$ in the ball $B$, that is
\begin{equation}
\begin{cases}
L_\Delta \psi_B = 1 \quad &\textrm{in} \ B,\\
\psi_B = 0 \quad &\textrm{in} \ \R^N \setminus B.
\end{cases}
\end{equation}
Let also $\eta \in C^\infty_c(B')$ be nonnegative, with $\max_{B'} \eta=1$ and $\int_{B'} \eta\,dx = \sigma>0$, and let $\tilde \eta(x):=\eta(\bar x)$. Notice that it is possible to choose $\eta$ in such a way that $\sigma=|B'|/2$. We introduce the antisymmetric function
\begin{equation}
w:= \psi_B - \psi_{\Q(B)} + \alpha \eta - \alpha \tilde \eta,
\end{equation}
where $\alpha > 0$ is a constant which will be chosen later on in the proof. Clearly $w \in \HH(\R^N)$; our goal is to show that 
\begin{equation}
\E_L (w, \phi) \le \int_{B} V w \phi \, dx \quad \textrm{for every nonnegative} \ \phi \in \HH(B).
\end{equation}
Given such a $\phi$, since $L_\Delta$ is a linear operator we have
\begin{equation}
\label{PS1s3eq45}
\E_L(w,\phi) = \E_L(\psi_B,\phi) - \E_L(\psi_{\Q(B)},\phi) + \alpha \E_L(\eta,\phi) - \alpha \E_L(\tilde \eta,\phi).
\end{equation}
Clearly, by definition
\begin{equation}
\label{PS1s3eq46}
\E_L(\psi_B,\phi) = \int_{B} \phi \, dx.
\end{equation}
The second term yields
\begin{equation}
\label{PS1s3eq47}
\begin{split}
\E_L(\psi_{\Q(B)},\phi) &= \frac{1}{2} \iint_{\R^N \times \R^N} (\psi_{\Q(B)}(x) - \psi_{\Q(B)}(y))(\phi(x) - \phi(y)) k(x-y) \, dx dy\\
&\qquad -  \iint_{\R^N \times \R^N} \psi_{\Q(B)}(x) \phi(y) j(x-y) \, dx dy \ + \rho_N \int_{\R^N} \psi_{\Q(B)} \phi \, dx \\
&=-  \iint_{\Q(B) \times B}  \psi_{\Q(B)}(x) \phi(y) k(x-y) \, dx dy\\
& \qquad - \iint_{\Q(B) \times B} \psi_{\Q(B)}(x) \phi(y) j(x-y) \, dx dy.
\end{split}
\end{equation}
In a similar way, we also obtain
\begin{equation}
\label{PS1s3eq48}
\begin{split}
\E_L(\eta,\phi) &= - \iint_{B' \times B} \eta(x) \phi(y) k(x-y) \, dx dy - c_N\iint_{B' \times B} \eta(x)\phi(y) j(x-y) \, dx dy,
\end{split}
\end{equation}
and
\begin{equation}
\label{PS1s3eq49}
\begin{split}
\E_L(\chi_{\Q(K)},\phi) &= - \iint_{\Q(B') \times B} \tilde \eta(x) \phi(y) k(x-y) \, dx dy - \iint_{\Q(B') \times B} \tilde \eta(x) \phi(y) j(x-y) \, dx dy\\
&=-  \iint_{B'\times B} \eta(x) \phi(y) k(\overline{x}-y) \, dx dy - \iint_{B' \times B} \eta(x)\phi(y) j(\overline{x}-y) \, dx dy.
\end{split}
\end{equation}

Plugging \eqref{PS1s3eq46}, \eqref{PS1s3eq47}, \eqref{PS1s3eq48} and \eqref{PS1s3eq49} into \eqref{PS1s3eq45}, we obtain
\begin{equation}
\begin{split}
\E_L(w,\phi) &= \int_{B} \phi \, dx+ c_N \iint_{\Q(B) \times B} \frac{\psi_{\Q(B)}(x) \phi(y)}{|x-y|^N} \, dx dy \\
& \qquad - \alpha c_N \int_B \phi(y) \int_{B'} \eta(x) \left( \frac{1}{|x - y|^N} - \frac{1}{|\overline{x} - y|^N} \right) \, dx \, dy \\
&\le (\kappa - \alpha C_1) \int_B \phi \, dx,
\end{split}
\end{equation}
where 
\begin{equation}
\kappa := 1 + c_N \, |B| \sup_{x \in \Q(B), \ y \in B} |x-y|^{-N} \  \sup_B \psi_B < +\infty, 
\end{equation}
and
\begin{equation}
C_1 := c_N \,  \inf_{x\in B', \, y \in B} \left( \frac{1}{|x - y|^N} - \frac{1}{|\overline{x} - y|^N} \right) \frac{|B'|}{2} >0
\end{equation}
The facts that $\kappa<+\infty$ and $C_1>0$ are justified, since $B'$ and $Q(B)$ has positive distance from $B$, $B'$ has positive distance from $T$, and any $x \in B'$ is closer to $y$ than its reflection $\overline{x}$.

As a consequence
\[
\E_L(w,\phi) - \int_{B} V w \phi\,dx \le (\kappa - \alpha C_1 + C_2) \int_{B}\phi \, dx,
\]
where $C_2:= \|V\|_{L^\infty(\Omega)} \sup_{B} \psi$, so that for $\alpha$ large enough we have that
\[
L_\Delta w \le V(x) w \qquad \text{in }\Omega.
\]
Moreover, $w$ is antisymmetric and $w = 0$ in $H \setminus (B \cup B')$. Since $\essinf_{B'} u>0$, we can now set $v:= u - \tau  w$ with $\tau:= \essinf_{B'} u / \alpha > 0$, and observe that $v$ is antisymmetric with respect to $T$, and satisfies
\begin{equation}
\begin{cases}
L_\Delta v \ge V(x) v \quad &\textrm{in} \ B,\\
v = u - \tau \alpha \eta  \ge 0 \quad &\textrm{in} \ H\setminus B.
\end{cases}
\end{equation}
Recalling that $\|V\|_{L^\infty(\Omega)} <\lambda_1^L(B)$, by Proposition \ref{lem: wmp} we obtain that
\begin{equation}
u \ge \tau \psi_B \quad \textrm{in} \ B.
\end{equation}
Thanks to \cite[Theorem 1.2]{hernandez2024optimal} we know that there exists a constant $c > 0$ depending on $N$ and $B$ such that
\begin{equation}
c^{-1} \, \ell^{1/2}(\dist (x,\partial B)) \le \psi_B (x) \le c \, \ell^{1/2}(\dist (x,\partial B)),
\end{equation}
and the thesis follows.
\end{proof}

An immediate consequence of the previous result is an antisymmetric strong maximum principle for $L_\Delta$.

\begin{corollary}[Antisymmetric Strong Maximum Principle for $L_\Delta$]
\label{anti_strong_max_pinciple_log_Lap}

Let $\Omega \subset H$ be a bounded open set with Lipschitz boundary, and let $u \in C(\R^N) \cap \HH(\R^N)$ be an antisymmetric function with respect to $T$ such that
\begin{equation}
\begin{cases}
    L_\Delta u \ge V(x) u \quad &\textrm{in} \ \Omega,\\
u \ge 0 \quad &\textrm{in} \ H,
\end{cases}
\end{equation}
where $V \in L^\infty(\Omega)$. Then either $u > 0$ in $\Omega$ or $u \equiv 0$ in $\R^N$.
\end{corollary}

\begin{proof}
Assume $u \not\equiv 0$ in $\R^N$. Then there exists a ball $B' \ssubset H$ such that $\inf_{B'} u > 0$. Let now $x \in \Omega \setminus {B'}$ and $B=B_r(x)$ a ball centered in $x$ with $r$ so small that $B \ssubset \Omega \cap H$, $B \ssubset H \setminus B'$, and $\lambda_1^L(B) > \|V\|_{L^\infty(\Omega)}$. It is always possible to choose such a $r$, thanks to the lower bound recalled in Remark \ref{lowerboundlambdaone}. We can then apply Lemma \ref{antisymmetric_Hopf_logarithmic_Laplacian} to $u$ on $B$, and obtain in particular that $u > 0$ in $B$. Since we can choose $x \in \Omega \setminus B'$ arbitrarily, the thesis follows.
\end{proof}

As a final result, we prove a Hopf-type lemma for antisymmetric functions at points \emph{on the symmetry hyperplane}. The key observation is that by looking at interior points of the domain we are able to retrieve informations on the partial derivative of order $1$ of the solution and therefore deduce a linear growth from the reflection hyperplane. In the context of the fractional Laplacian, such a result was first obtained in \cite{soave2019overdetermined} (see also \cite{dipierro2024role}). We need a preliminary statement.

\medskip

In what follows, we set $H^+:= \{ x_1 > 0\}$, $B_r$ the ball of radius $r > 0$ centered at the origin, $B_r^+ := B_r \cap H^+$ and for a given $e \in \mathcal{S}^{N-1}$ we say that a function is \emph{$e$-antisymmetric} if it is antisymmetric with respect to $T_e := \{ x \cdot e = 0\}$.

\begin{lemma}
\label{interior_anti_Hopf_lemma_log_lap}
Let $r>0$, and let $V\in L^\infty(B_{2r}^+)$. Then there exists a $e_1$-antisymmetric function $\phi \in C^\infty_c(\R^N)$ such that
\begin{equation}
\begin{cases}
L_\Delta \phi \le V(x) \phi \quad &\textrm{in} \ B_{2r}^+ \setminus B_{r/2} (r e_1),\\
\phi = 0 \quad &\textrm{in} \ H^+ \setminus B_{2r}^+,\\
\phi \le 1 \quad &\textrm{in} \ B_{r/2} (re_1),\\
\partial_1 \phi (0) > 0,
\end{cases}
\end{equation}
in classical (i.e. pointwise) sense.
\end{lemma}

\begin{proof}
In the proof we make use of two properties of $L_\Delta u$ when $u$ is an antisymmetric function. We assume that the pointwise expression of the Logarithmic Laplacian, Equation \eqref{pointwise L D}, is available. First we notice that the Logarithmic Laplacian of an antisymmetric function is antisymmetric. Indeed, let $u$ be an antisymmetric function; then, for every $x \in \R^N$
\begin{equation}
\begin{aligned}
L_\Delta u(\overline{x}) &= c_N \int_{\R^N} \frac{u(\overline{x}) \chi_{B_1(\overline{x})}(y) - u(y)}{|\overline{x} - y|^N} \, dy + \rho_N u(\overline{x})\\
&= c_N \int_{\R^N} \frac{-u(x) \chi_{B_1(x)}(\overline{y}) - u(y)}{|x - \overline{y}|^N} \, dy - \rho_N u(x)\\
&= c_N \int_{\R^N} \frac{-u(x) \chi_{B_1(x)}(y) - u(\overline{y})}{|x - y|^N} \, dy - \rho_N u(x)\\
&= - c_N \int_{\R^N} \frac{u(x) \chi_{B_1(x)}(y) - u(y)}{|x - y|^N} \, dy - \rho_N u(x) = - L_\Delta u(x).
\end{aligned}
\end{equation}
Moreover, we claim that
\begin{equation}
\label{log_Lap_anti_form}
\begin{aligned}
L_\Delta u(x) &= c_N \int_H \left( \frac{1}{|x-y|^N} - \frac{1}{|\overline{x}-y|^N} \right) \left( u(x) \chi_{B_1(x)}(y) - u(y) \right) \, dy\\
&+ c_N u(x) \int_H \frac{\chi_{B_1(x)}(y) + \chi_{B_1(\overline{x})}(y)}{|\overline{x}-y|^N} \, dy + \rho_N u(x).
\end{aligned}
\end{equation}

Indeed, since $u$ is antisymmetric
\begin{equation}
\begin{aligned}
\int_{H^c} \frac{u(x) \chi_{B_1(x)}(y) - u(y)}{|x - y|^N} \, dy &= \int_{H} \frac{u(x) \chi_{B_1(x)}(\overline{y}) - u(\overline{y})}{|x - \overline{y}|^N} \, dy\\
& = \int_{H} \frac{u(x) \chi_{B_1(\overline{x})}(y) + u(y)}{|\overline{x} - y|^N} \, dy .
\end{aligned}
\end{equation}
Therefore
\begin{equation}
\begin{aligned}
L_\Delta u(x) &= c_N \int_{H} \frac{u(x) \chi_{B_1(x)}(y) - u(y)}{|x - y|^N} \, dy + c_N \int_{H} \frac{u(x) \chi_{B_1(\overline{x})}(y) + u(y)}{|\overline{x} - y|^N} \, dy + \rho_N u(x),
\end{aligned}
\end{equation}
whence \eqref{log_Lap_anti_form} follows.

Let now $\zeta$ be a smooth $e_1$-antisymmetric function with compact support in $B_{2r}$ such that $\zeta \ge 0$ in $\R^N_+$, and $\partial_1 \zeta(0)>0$. Since $\zeta \in C^\infty_c(\R^N)$, it is not difficult to check that $L_\Delta \zeta \in C^\infty(\R^N)$: indeed, if $u \in C^{1,\alpha}_c(\R^N)$ (actually $C^{1,\mathrm{Dini}}_c(\R^N)$ would be sufficient), partial derivatives commute with $L_{\Delta}$; and at this point the fact that $L_\Delta u$ is of class $C^1$ follows from \cite[Proposition 2.2]{chen2019dirichlet}.

Therefore, recalling also that $L_\Delta \zeta$ is antisymmetric, there exists $\tilde{C}_1 > 0$ such that
\begin{equation}
L_\Delta \zeta (x) -V(x) \zeta (x) \le \tilde{C}_1 x_1 \quad \textrm{in} \ B_{2r}^+.
\end{equation}

Let also $\tilde{\zeta}$ be a smooth $e_1$-antisymmetric function such that $\tilde{\zeta} \equiv 1$ in $B_{r/4}(re_1)$, $\tilde{\zeta} \equiv 0$ in $B_{3r/8}(re_1)^c$, and $0 \le \tilde{\zeta} \le 1$ in $\R^N_+$. From \eqref{log_Lap_anti_form}, for every $x \in B_{2r}^+ \setminus B_{r/2}(re_1)$ we have
\begin{equation}
L_\Delta \tilde{\zeta} (x) = - c_N \int_H \left( \frac{1}{|x-y|^N} - \frac{1}{|\overline{x}-y|^N} \right) \tilde{\zeta}(y) \, dy.
\end{equation} 
Notice that, for every $y \in B_{3r/8}(re_1)$  
\begin{equation}
\begin{aligned}
\frac{1}{|x-y|^N} - \frac{1}{|\overline{x}-y|^N} = \frac{N}{2} \int_{|x-y|^2}^{|\overline{x}-y|^2} \frac{dt}{t^{\frac{N+2}{2}}} \ge \frac{4 x_1 y_1}{|\overline{x}-y|^{N+2}} \ge \tilde{C}_2 x_1,
\end{aligned}
\end{equation}
with $\tilde{C}_2>0$ independent of $y \in B_{3r/8}(e_1)$, and therefore 
\begin{equation}
L_\Delta \tilde{\zeta} (x) - V(x) \tilde{\zeta}(x) = L_\Delta \tilde{\zeta} (x)  \le - \tilde{C}_3 x_1 \qquad \text{in $B_{2r}^+ \setminus B_{r/2} (re_1)$}.
\end{equation}

At this point it is not difficult to check that the function $\phi:= \alpha(\zeta + M \tilde{\zeta})$, where $M > 0$ is chosen so large that $\tilde{C}_1 - M \tilde{C}_3 < 0$, and $\alpha>0$ is chosen so small that $\alpha(\|\zeta\|_{L^\infty(B_{2r}^+)} +M ) \le 1$, satisfy all the requirement of the thesis.
\end{proof}

\begin{lemma}[Hopf-type lemma on the symmetry hyperplane]
Let $u \in \HH(\R^N) \cap C(B_R^+)$ be an $e_1$-antisymmetric function such that
\begin{equation}
\begin{cases}
    L_\Delta u \ge V(x) u  \quad &\textrm{in} \ B_R^+,\\
    u>0 & \text{in} \ B_R^+, \\
u \ge 0 \quad &\textrm{in} \ H^+,
\end{cases}
\end{equation}
in weak sense, with $V \in L^\infty(B_R^+)$. Then, for every $\eps>0$ and $r>0$ sufficiently small, we have that
\[
u \ge \eps \phi \qquad \text{in $B_{2r}^+$},
\]
where $\phi$ is given by Lemma \ref{interior_anti_Hopf_lemma_log_lap}. In particular
\[
\liminf_{h \to 0^+} \frac{u(he_1)}{h}>0.
\]
\end{lemma}

\begin{proof}
Let $V_\infty:=\|V\|_{L^\infty(B_R^+)} \ge 0$, and let $\delta=\delta(V_\infty)>0$ be given by the maximum principle in small sets (Corollary \ref{cor: awmp small domains}). We choose $r \in (0,1/2)$ so small that $|B_{2r} \setminus B_{r/2}(re_1)|<\delta$, and take $\phi$ as in Lemma \ref{interior_anti_Hopf_lemma_log_lap}. Since $u>0$ in $B_R^+$ and $B_{r/2}(re_1) \subset \subset B_R^+$, we have that $u-\eps \phi \ge 0$ in $\overline{B_{r/2}(re_1)}$, provided that $\eps>0$ is sufficiently small. Hence $u -\eps \phi \ge 0$ in $(\{x_1>0\} \setminus B_{2r}^+) \cup \overline{B_{r/2}(re_1)}$, and moreover
\[
L_\Delta (u -\eps \phi)- V(x)(u-\eps \phi) \ge 0 \quad \text{in } B_{2r}^+ \setminus B_{r/2}(re_1),
\]
in weak sense\footnote{By Lemma \ref{interior_anti_Hopf_lemma_log_lap}, we know that $\phi$ is a pointwise subsolution of $L_\Delta \varphi \le V(x) \phi$. Since moreover it is in $C^\infty_c(\R^N)$, it is also a weak solution. In fact, much less is needed, see \cite[Remark 4.6]{chen2019dirichlet} for more details.}. By Corollary \ref{cor: awmp small domains}, we deduce that $u \ge \eps \phi$ in $B_{2r}^+$, for every $\eps>0$ small. The thesis follows easily.
\end{proof}

\subsection{Further maximum principles and Hopf-type lemmas}

The proofs of Proposition \ref{lem: wmp}, Lemma \ref{antisymmetric_Hopf_logarithmic_Laplacian}, and Corollary \ref{anti_strong_max_pinciple_log_Lap} can be easily modified to remove the antisymmetry assumption. Clearly, one obtains results which are not so relevant in the application of the moving planes method, but they may be useful in different contexts. For this reason, we state the results for the reader's convenience. Notice that they differ from those in \cite{chen2019dirichlet} and \cite{hernandez2024optimal}, since they include possibly sign-changing potentials $V$. Moreover, we do not require the connectedness of $\Omega$.

 \begin{proposition}[Weak maximum principle]
\label{lem: wmp na}
Let $\Omega \subset \R^N$ be a bounded open set with Lipschitz boundary, and let $u \in \HH(\R^N)$ be such that
\begin{equation}
\begin{cases}
    L_\Delta u \ge V(x) u \quad &\textrm{in} \ \Omega,\\
u \ge 0 \quad &\textrm{in} \ \R^N \setminus \Omega
\end{cases}
\end{equation}
in a weak sense, where $V \in L^\infty(\Omega)$ is such that $\|V\|_{L^\infty(\Omega)} \le \lambda^L_1(B)$. Then $u \ge 0$ a.e. in $\Omega$.
\end{proposition}
\begin{proof}
The proof is analogue to the one of Proposition \ref{lem: wmp na}, by taking $\varphi=u_-$.
\end{proof}

\begin{corollary}[Maximum principle in sets with small measure]\label{cor: wmps}
Let $V_\infty>0$. There exists $\delta>0$ depending only on $V_\infty$ and on the dimension $N$ such that, if $\Omega \subset H$ is a bounded open set of $\R^N$ with Lipschitz boundary, if $|\Omega|<\delta$, and if $u \in \HH(\R^N)$ satisfies
\begin{equation}
\begin{cases}
    L_\Delta u \ge V(x) u \quad &\textrm{in} \ \Omega,\\
u \ge 0 \quad &\textrm{in} \ \R^N \setminus \Omega,
\end{cases}
\end{equation}
where $V \in L^\infty(\Omega)$ and $\|V\|_{L^\infty(\Omega)} \le V_\infty$, then $u \ge 0$ a.e. in $\Omega$.
\end{corollary}

\begin{lemma}[Hopf Lemma for $L_\Delta$]
\label{Hopf_logarithmic_Laplacian}
Let $\Omega \subset \R^N$ be a bounded open set with Lipschitz boundary, let $B \subset \Omega$ be a ball with $\lambda_1^L(B)>0$, and let $u \in \HH(\R^N)$ satisfy
\begin{equation}
\begin{cases}
    L_\Delta u \ge V(x) u  \quad &\textrm{in} \ \Omega,\\
u \ge 0 \quad &\textrm{in} \ \R^N,
\end{cases}
\end{equation}
where $V \in L^\infty(\Omega)$ is such that $\|V\|_{L^\infty(\Omega)} \le \lambda^L_1(B)$.
Let also $B' \ssubset \R^N \setminus \overline{B}$ be another ball, such that $\essinf_{B'} u > 0$. Then, there exists a constant $C > 0$ depending on $N$, $B$, $\dist(B',B)$, $\essinf_{B'} u$ and $\| V \|_{L^\infty(\Omega)}$, such that
\[
u(x) \ge C \, \ell^{1/2} (\dist (x,\partial B)) \quad \textrm{for a.e.} \ x \in B,
\]
where $\ell$ is defined in \eqref{def ell}. Moreover, if $u\in C(\overline{B})$ and $u(x_0) = 0$ for some $ x_0 \in \partial B \cap \partial \Omega$, then we have
\begin{equation}
\liminf_{t \searrow 0} \frac{u(x_0 - t \nu(x_0)) }{\ell^{1/2}(t)} > 0,
\end{equation}
where $\nu$ is the outer unit normal vector field along $\partial B$.
\end{lemma}

\begin{proof}
We proceed as in Lemma \ref{antisymmetric_Hopf_logarithmic_Laplacian}, by taking as comparison function 
\[
w:= \psi_B + \alpha \eta. \qedhere
\]
\end{proof}

\begin{corollary}[Strong Maximum Principle for $L_\Delta$]
\label{strong_max_pinciple_log_Lap}
Let $\Omega \subset H$ be a bounded open set with Lipschitz boundary, and let $u \in C(\R^N) \cap \HH(\R^N)$ satisfy
\begin{equation}
\begin{cases}
    L_\Delta u \ge V(x) u \quad &\textrm{in} \ \Omega,\\
u \ge 0 \quad &\textrm{in} \ \R^N,
\end{cases}
\end{equation}
where $V \in L^\infty(\Omega)$. Then either $u > 0$ in $\Omega$, or $u \equiv 0$ in $\R^N$.
\end{corollary}

\begin{proof}
We proceed as in Corollary \ref{anti_strong_max_pinciple_log_Lap}, by using Lemma \ref{Hopf_logarithmic_Laplacian} instead of Lemma \ref{antisymmetric_Hopf_logarithmic_Laplacian}.
\end{proof}

In the appendix, we discuss what changes in \cite{chen2019dirichlet} if we consider a bounded open set $\Omega$ instead of a domain in Sections 2 and 3. We point out that the only property where the connectedness has a role is the strict positivity (or strict negativity) of any first eigenfunction. Instead, the fact that any first eigenfunction does not change sign still holds in general open bounded sets, with the same proof as in \cite{chen2019dirichlet}. Regarding the strict positivity (or strict negativity) of any first eigenfunction, this point follows now from Corollary \ref{strong_max_pinciple_log_Lap} also in general open bounded sets.

\section{Symmetry of solutions in symmetric domains}\label{sec: GNN}

In this section we prove the Gidas-Ni-Nirenberg symmetry result for the Logarithmic Laplacian, adapting the proof by Berestycki-Nirenberg \cite{BeNi}. Having established weak and strong maximum principles in the previous section, the proof is standard, but we present it for the sake of completeness.

\medskip

We introduce the notation needed for the application of the method of moving planes, which we will use both here and in the following section. Given a bounded and smooth enough set $E \subset \R^N$, a unit vector $e \in \mathbb{S}^{N-1}$ and a parameter $\lambda \in \R$, we define
\begin{align*}
    &T_\lambda = T_\lambda^e = \{ x \in \R^N \, | \, x \cdot e = \lambda \} & &\textrm{a hyperplane orthogonal to } e,\\
    &H_\lambda = H_\lambda^e = \{ x \in \R^N \, | \, x \cdot e > \lambda \} & &\textrm{the ``positive'' half space with respect to } T_\lambda, \\
    &E_\lambda = E \cap H_\lambda & &\textrm{the ``positive'' cap of } E,\\
    &x_\lambda = x -2(x\cdot e - \lambda) \, e & &\textrm{the reflection of } x \textrm{ with respect to } T_\lambda,\\
    &\mathcal{Q}_\lambda = \mathcal{Q}_\lambda^e : \R^N \to \R^N, x \mapsto x_\lambda & &\textrm{the reflection with respect to } T_\lambda.
\end{align*}

When there is no chance of ambiguity, the dependence on the unit vector $e$ in the notation will be dropped. If $E \subset \mathbb{R}^N$ is an open bounded set with Lipschitz boundary, it makes sense to define
$$
\Lambda_e := \sup \{ \ x \cdot e \ | \ x \in E \ \}
$$ 
and
\begin{equation*}
    \lambda_e = \inf \{ \ \lambda \in \R \ | \ \mathcal{Q}_{\tilde{\lambda}}(E_{\tilde{\lambda}}) \subset E, \textrm{for all} \ \Tilde{\lambda} \in (\lambda, \Lambda_e) \ \}.
\end{equation*}

{F}rom this point on, given a direction $e \in \mathbb{S}^{N-1}$, we refer to $T_{\lambda_e} = T^e$ and $E_{\lambda_e} = \widehat{E}$ as the \textit{critical hyperplane} and the \textit{critical cap} with respect to $e$, respectively, and call $\lambda_e$ the \textit{critical value} in the direction $e$. If $\partial E$ is of class $C^1$, we recall from \cite{serrin1971symmetry} (see also \cite[Appendix A]{AmFr}) that, for any given direction $e$, at least one of the following two conditions holds:

\medskip

\textbf{Case 1} - The boundary of the reflected cap $\mathcal{Q}(\widehat{E})$ becomes internally tangent to the boundary of $E$ at some point $P \not \in T$;

\medskip

\textbf{Case 2} - the critical hyperplane $T$ becomes orthogonal to the boundary of $E$ at some point $Q \in T$.

\medskip

We can now prove the symmetry of solutions in symmetric sets..


\begin{proof}[Proof of Theorem \ref{GNN_convex_domain}]
Let $x=(x_1,x') \in \R \times \R^{N-1}$ and let $\Lambda:= \sup \{ \ x_1 \ | \ x \in \Omega \ \}$. 
Under our assumptions, it is plain that $\lambda_{e_1}=0$, and, for $\lambda \in (0,\Lambda)$, we define
$$
w_\lambda (x) := u(x^\lambda) - u(x) \quad \textrm{for} \ x \in \Sigma_\lambda:= \Omega \cap H_\lambda.
$$
We can easily see that $w_\lambda$ is antisymmetric with respect to $T_\lambda$ and that
\begin{equation}
\begin{cases}
L_\Delta w_\lambda + c_\lambda(x) w_\lambda = 0 \quad &\textrm{in} \ \Sigma_\lambda,\\
w_\lambda \ge 0 \quad &\textrm{in} \ H_\lambda \setminus \Sigma_\lambda,
\end{cases}
\end{equation}
where
\begin{equation}
c_\lambda(x):=
\begin{cases}
- \frac{ f(u(x^\lambda)) - f(u(x))}{u(x^\lambda) - u(x)} \quad &\textrm{if} \ u(x^\lambda) \neq u(x),\\
0 \quad  &\textrm{if} \ u(x^\lambda) = u(x).
\end{cases}
\end{equation}
Clearly, $c_\lambda \in L^\infty (\Sigma_\lambda)$ with its $L^\infty$ norm bounded by a constant $c_\infty$ independent of $\lambda$ (depending only on $\|u\|_{L^\infty(\Omega)}$ and on the Lipschitz constant of $f$ on $[0,\|u\|_{L^\infty(\Omega)}]$). Let $\delta=\delta(c_\infty,N)$ be given by the weak maximum principle in sets with small measure, Corollary \ref{cor: awmp small domains}. Then, any $\lambda \in (0, \Lambda)$ close enough to $\Lambda$ is such that $|\Sigma_\lambda| <\delta$. We can therefore apply Corollary \ref{cor: awmp small domains}, and obtain that
$$
w_\lambda \ge 0 \quad \textrm{in} \ \Sigma_\lambda,
$$
for every $\lambda \in (\Lambda-\eps,\Lambda)$ with $\eps$ small and positive. By the strong maximum principle we then have that either $w_\lambda > 0$ in $\Sigma_\lambda$, or $w_\lambda \equiv 0$ in $\Sigma_\lambda$; since $w_\lambda$ is continuous up to the boundary and $w_\lambda = 0$ on $\partial \Sigma_\lambda$, the latter cannot hold.

We have then proved that the set
$$
\{ \lambda \in (0, \Lambda) \ | \ w_\nu > 0 \ \textrm{in} \ \Sigma_\nu \ \textrm{for every} \ \nu \in (\lambda, \Lambda)\}
$$
is non-empty. Let $\lambda_\star$ be its infimum. We want to show that $\lambda_\star = 0$. 

\medskip

Assume by contradiction that $\lambda_\star > 0$. We know that $w_{\lambda_\star} \ge 0$ in $\Sigma_{\lambda_\star}$, by continuity; as before, the strong maximum principle together with the continuity of $w_{\lambda_\star}$ yields $w_{\lambda_\star} > 0$ in $\Sigma_{\lambda_\star}$.

\medskip

Let now $K \subset \subset \Sigma_{\lambda_\star}$ be a compact set with $| \Sigma_{\lambda_\star} \setminus K | < \delta/2$, where $\delta=\delta(c_\infty,N)$ was defined before. By continuity, there exists $\eta > 0$ such that $\inf_K w_{\lambda_\star} \ge \eta$. We can then choose $\varepsilon > 0$ small enough so that $\inf_K w_{\lambda_\star - \varepsilon} \ge \eta/2$ and $| \Sigma_{\lambda_\star - \varepsilon} \setminus K | < \delta$. Thus, the weak maximum principle gives
$$
w_{\lambda_\star - \varepsilon} \ge 0 \quad \textrm{in} \ \Sigma_{\lambda_\star - \varepsilon} \setminus K,
$$
which yields, together with the lower bound on $K$ and the strong maximum principle,
$$
w_{\lambda_\star - \varepsilon} > 0 \quad \textrm{in} \ \Sigma_{\lambda_\star - \varepsilon}
$$ 
for every $\eps>0$ small enough, which is a contradiction of the minimality of $\lambda_\star$. Therefore, $\lambda_\star = 0$, and by continuity again 
\begin{equation}
\label{lksajdlkjaslk}
u(-x_1,x') \ge u(x_1, x') \quad \textrm{for every} \ x \in \Omega \cap \{ x_1 > 0 \}.
\end{equation}
By repeating the same argument in the direction $-e_1$ we obtain equality in \eqref{lksajdlkjaslk}, that is, $u$ is symmetric with respect to $\{ x_1 = 0 \}$. The monotonicity easily follows from the argument.
\end{proof}

\section{Parallel surface Logarithmic Torsion}\label{sec: parallel}

%

\begin{proof}[Proof of Theorem \ref{parallel_surface_log_torsion}]
We apply the method of moving planes on the set $G$ with respect to direction $e_1 \in \mathbb{S}^{N-1}$. Assume, without loss of generality, that the critical position for $G$ is reached for $\lambda = 0$, and recall that $\Lambda= \sup \{ \ x_1 \ | \ x \in \Omega \ \}$. As already observed in previous contributions on the parallel surface problem (see e.g. \cite[Section 3]{ciraolo2023symmetry}), the critical position for $G$ is also the critical position for $\Omega$. Therefore, for $\lambda \in (0,\Lambda)$, we define
$$
w_\lambda (x) := u(x^\lambda) - u(x) \quad \textrm{for} \ x \in \Sigma_\lambda:= \Omega \cap H_\lambda.
$$
Just like in the proof of Theorem \ref{GNN_convex_domain}, we can show that $w_\lambda \ge 0$ in $\Sigma_\lambda$ for $\lambda \in (0,\Lambda)$, $\lambda$ close enough to $\Lambda$, thanks to the antisymmetric weak maximum principle in sets with small measure (Corollary \ref{cor: awmp small domains}). Let then
$$
\lambda_\star:= \inf \{ \ \lambda \in (0, \Lambda) \ | \ w_\nu > 0 \ \textrm{in} \ \Sigma_\nu \ \textrm{for every} \ \nu \in (\lambda, \Lambda) \ \}.
$$
It is easy to see that $\lambda_\star = 0$. Indeed, if this were not the case we would have 
$$
w_{\lambda_\star} \ge 0 \quad \textrm{in} \ \Sigma_{\lambda_\star}.
$$ 
From Corollary \ref{anti_strong_max_pinciple_log_Lap} we would then have that either $w_{\lambda_\star} > 0$ in $\Sigma_{\lambda_\star}$ or $w_{\lambda_\star} \equiv 0$ in $\R^N$, but the latter cannot occur since $w_{\lambda_\star} > 0$ on $\partial \Sigma_{\lambda_\star}$ which would in turn contradict the minimality of $\lambda_\star$. Therefore, $\lambda_\star = 0$.

\medskip

We have now reached the critical position and want to make use of the overdetermined condition \eqref{parallel_surface_condition} to prove the result by showing that $w_0 \equiv 0$ in $\Sigma_0$. Assume by contradiction this is not the case; since we are in the critical position, we know that at least one of two possible cases occur for the critical cap $\widehat G$.

\medskip

\textbf{Case 1} - The boundary of the reflected cap $\mathcal{Q}(\widehat{G})$ becomes internally tangent to the boundary of $G$ at some point $P \not \in T$. In this case, we immediately get that
\begin{equation}
w(P) = u(\mathcal{Q}(P)) - u(P) = c - c = 0,
\end{equation}
which is already a contradiction.\\

\textbf{Case 2} - The critical hyperplane $T$ becomes orthogonal to the boundary of $G$ at some point $Q \in T$. Without loss of generality, we can assume that $Q=0$ and $T = \{ x_1 = 0\}$. On the other hand, this is in contradiction with Lemma \ref{interior_anti_Hopf_lemma_log_lap}, since $u \ge \eps \phi$ in a neighborhood of $0$, with $\partial_1 \phi(0)>0$.

\medskip

We then have that $w_0 = 0$ in $\Sigma_0$, that is $G$ and $\Omega$ are symmetric with respect to direction $e_1$ and so is $u$. We can then repeat the same argument with respect to any direction $e \in \mathbb{S}^{N-1}$ which leads to the desired result.
\end{proof}

\appendix 
\section{Remarks on spectral properties of the Logarithmic Laplacian in open bounded sets}

In \cite{chen2019dirichlet}, the authors studied the eigenvalue problem for the Logarithmic Laplacian with homogeneous Dirichlet exterior data on a bounded \emph{domain} $\Omega \subset \R^N$. In this appendix, we briefly discuss what changes if the assumption of the connectedness of $\Omega$ is removed. Actually, not much changes. All the results in Section 2, and Theorem 3.1, Proposition 3.2, and Lemma 3.3 of \cite{chen2019dirichlet} hold also for non-connected sets, with minor or no changes (notice in particular that the Poincar\'e inequality, \cite[Lemma 2.7]{FeKaVo}, and the compactness of the embedding of $\mathbb{H}(\Omega)$ into $L^2(\Omega)$, \cite[Theorem 1.2]{CoDP}, two results quoted in Section 3 of \cite{chen2019dirichlet}, hold true in general open sets). Regarding Theorem 3.4, the existence of an increasing sequence of eigenvalues $\{\lambda_k^L(\Omega)\}$ tending to $+\infty$, their variational characterization, and the fact that the associated eignefunctions form a Hilbert basis of $L^2(\Omega)$, can be proved also for non-connected open bounded sets, without changing the proof. The only point which requires some care is the strict positivity of the first eigenfunction, and the fact that it is simple. In this point the authors used \cite[Theorem 1.1]{JaWe2019} to deduce that any first eigenfunction has a strict sign, which exploits the connectedness of $\Omega$. However, since the variational characterization of $\lambda_1^L(\Omega)$ and \cite[Lemma 3.3]{chen2019dirichlet} are valid in general open bounded sets, we can still deduce that any first eigenfunction $w$ does not change sign. This is sufficient to show that there exists a unique nonnegative $L^2$-normalized eigenfunction, as in \cite{chen2019dirichlet}. To sum up, the only property of \cite[Theorem 3.4]{chen2019dirichlet} which may fail, or in any case really requires a different proof with respect to \cite{chen2019dirichlet}, is the strict poisitivity of the first eigenfunction. At this point \cite[Theorem 3.5]{chen2019dirichlet} can be stated and proved in general open bounded sets without changing the proof (the proof does not use the strict positivity of the first eigenfunction), and, finally \cite[Corollary 3.6]{chen2019dirichlet} (the Faber-Krahn inequality) holds for general open bounded sets as well. Notice that in \cite{chen2019dirichlet} refer to \cite{BaLaMH} for the Faber-Krahn inequality for the fractional Laplacian. The result in \cite{BaLaMH} is stated for connected sets. But the validity of the Faber-Krahn inequality for the fractional Laplacian was proved in \cite{BrLiPa} for general open bounded sets. Therefore, in the proof of \cite[Corollary 3.6]{chen2019dirichlet} for general open bounded sets one should replace \cite[Theorem 5]{BaLaMH} with \cite[Theorem 3.5]{BrLiPa}.

\bigskip

\noindent \textbf{Data availability.} Data sharing not applicable to this article as no datasets were generated or analysed during the current study.

\medskip

\noindent \textbf{Conflict of interest.} The authors declare that they have no conflict of interest.


\newcommand{\etalchar}[1]{$^{#1}$}

\end{document}